\documentclass[11pt]{article}%
\usepackage{mathrsfs}
\usepackage{bbm}
\usepackage{amsfonts}
\usepackage{amsmath,amssymb}
\usepackage{amsmath}
\usepackage{amssymb}
\usepackage{graphicx}%
\usepackage{shorttoc}
\usepackage{cite}

\providecommand{\U}[1]{\protect \rule{.1in}{.1in}}

\newtheorem{theorem}{Theorem}[section]
\newtheorem{corollary}[theorem]{Corollary}
\newtheorem{definition}[theorem]{Definition}
\newtheorem{example}[theorem]{Example}
\newtheorem{lemma}[theorem]{Lemma}
\newtheorem{proposition}[theorem]{Proposition}
\newtheorem{remark}[theorem]{Remark}

\newenvironment{proof}[1][Proof]{\noindent \textbf{#1.} }{\  $\Box$}
\numberwithin{equation}{section}

\begin{document}

\title{\textbf{Invariant Sublinear Expectations}}
\author{Yongsheng Song\thanks{RCSDS, Academy of Mathematics and Systems Science, Chinese Academy of Sciences, Beijing 100190, China, and School of Mathematical Sciences, University of Chinese Academy of Sciences, Beijing 100049, China.E-mails: yssong@amss.ac.cn (Y. Song).}  }

\date{\today}
\maketitle

\begin{abstract}

We first give a decomposition for a $T$-invariant sublinear expectation $\mathbb{E}=\sup_{P\in\Theta}\mathrm{E}_P$, and show that each component $\mathbb{E}^{(d)}=\sup_{P\in\Theta^{(d)}}\mathrm{E}_P$ of the decomposition has a finite period $p_d\in\mathbb{N}$, i.e., \[\mathbb{E}^{(d)}\left[f-f\circ T^{p_d}\right]=0, \quad f\in\mathcal{H}.\] 
Then we prove that a continuous invariant sublinear expectation that is strongly ergodic has a finite period $p_{\mathbb{E}}$, and each component $\Theta^{(d)}$ of its periodic decomposition is the convex hull of a finite set of $T^{p_d}$-ergodic probabilities. 

As an application of the characterization, we prove an ergodicity result which shows that the limit of the $p_{\mathbb{E}}$-step time means achieves the upper expectation.
  
\end{abstract}


\textbf{MSC-classification}: 28A12; 28D05; 37A30.

\section{Introduction}

The notions of invariant and ergodic capacities were initiated by \textsl{Cerreia-Vioglio, Maccheroni, and Marinacci} \cite{ergodictheorem}, in which they gave the first Birkhoff's ergodic theorem for continuous invariant capacities:

Let $(\Omega,\mathcal{F}, T)$ be a measurable system, and let $\mathbb{V}=\sup_{P\in\Theta}P$ be a continuous upper probability.  If  $\mathbb{V}$ is ergodic, i.e., $\mathbb{V}(A)\in \{0, 1\}$ for $T$-invariant set $A$, then for any $P\in\Theta$, one gets
\begin{eqnarray}\label{ergodicresult1}
    \begin{split}
        -\mathrm{C}_{\mathbb{V}}[-\xi^{*}]\leq\lim_{n\to\infty}\frac{1}{n}\sum_{k=0}^{n-1}\xi(T^{k}\omega)\leq \mathrm{C}_{\mathbb{V}}[\xi^{*}], \ \ P\mbox{-a.s.},
    \end{split}
\end{eqnarray}
where $\xi^{*}(\omega)=\limsup\limits_{n\to\infty}\frac{1}{n}\sum_{k=0}^{n-1}\xi(T^{k}\omega)$ and $\mathrm{C}_{\mathbb{V}}$ represents the Choquet integral w.r.t $\mathbb{V}.$

We refer the readers to \textsl{Feng and Zhao} \cite{FZ21} and \textsl{Feng, Wu, and Zhao} \cite {FWZ20} for ergodicity results of capacities and sublinear expectations under a different definition.

\textsl{Sheng and Song} \cite{Sheng-Song24} provided a characterization of the continuous ergodic capacity $\mathbb{V}$. Specifically, let $\Theta$ be the set of probabilities on $\Omega$ dominated by $\mathbb{V}$,
 and let $\Theta^{(1)}$ (resp. $\Theta^{(1)}_{\mathfrak{e}}$) be the subset of $T$-invariant (resp. ergodic) probabilities in $\Theta$. Then, the cardinality of $\Theta^{(1)}_{\mathfrak{e}}$ is finite, and $\Theta^{(1)}=\textrm{co}\ \Theta^{(1)}_{\mathfrak{e}}$, the convex hull of $\Theta^{(1)}_{\mathfrak{e}}$.

As a by-product, they improve the ergodic theorem (\ref{ergodicresult1}). Let $\mathbb{E}^{(1)}=\sup_{P\in\Theta^{(1)}}\mathrm{E}_{P}.$ Then, for any $P\in\Theta$ and any bounded random variable $\xi$, they show that
$$
    -\mathbb{E}^{(1)}[-\xi]\leq\lim_{n\to\infty}\frac{1}{n}\sum_{k=0}^{n-1}\xi(T^{k}\omega)\leq \mathbb{E}^{(1)}[\xi], \ \ P\mbox{-a.s.}
$$
 The limit is bounded by the upper expectation $\mathbb{E}^{(1)}[\xi]$, instead of the Choquet integral $\mathrm{C}_{\mathbb{V}}[\xi^{*}].$

Note that $\mathbb{E}^{(1)}[\xi]<\mathbb{E}[\xi]$ generally,  so a natural question is how to achieve the upper and lower expectations via the time means. The main purpose of this paper is  trying to answer this question by giving further characterizations for the structure of the set $\Theta$.

 Let $(\Omega,\mathcal{F}, T)$ be a measurable system, and let $\mathcal {H}$ be a vector space of $\mathcal{F}$-measurable functions on $\Omega$ such that $1\in \mathcal{H}, f\circ T, |f|\in \mathcal{H} \ \mathrm{if} \ f\in\mathcal{H}.$ A sublinear expectation $\mathbb{E}=\sup_{P\in \Theta}\mathrm{E}_P$ on $(\Omega, \mathcal{H})$ is said to be $T$-invariant if \[\mathbb{E}[f\circ T]=\mathbb{E}[f],  \ \emph{for} \ f\in\mathcal{H}.\] Here  $\Theta$ is a subset of $\mathcal{M}(\Omega)$, the collection of probabilities on $(\Omega, \mathcal{F})$.

In Section 3, we give a decomposition for the invariant sublinear expectation (ISE), and prove that each component $\mathbb{E}^{(d)}=\sup_{P\in\Theta^{(d)}}\mathrm{E}_P$ of the decomposition is periodic, i.e., there exists $p\in\mathbb{N}$ such that 
\[\mathbb{E}^{(d)}[f-f\circ T^p]=0, \quad f\in\mathcal{H}.\]
Write $p_d$ for the smallest such integer, which is called the period of the ISE $\mathbb{E}^{(d)}$. The main results in this section are Proposition \ref {perioD} and Theorem \ref {Thm-chara-PeriodicD}. 

We say an ISE $\mathbb{E}$ has a periodic decomposition if 
\begin {equation}\label {PD}\mathbb{E}[f]=\sup_{d\in\mathbb{N}}\mathbb{E}^{(d)}[f],  \quad f\in\mathcal{H}.
\end {equation}
A typical example with this property is the sublinear expectation $\mathbb{E}$ defined on $(\mathbb{R}^{\infty}, C_b(\mathbb{R}^{\infty}))$ endowed with the product topology, such that the sequence of random variables 
$\xi_n(\omega)=\omega_n$, $n\in\mathbb{N}$, are i.i.d. For this case,

\begin {itemize}
\item [$\cdot$] $\mathbb{E}$ has a periodic decomposition (\ref {PD}) with respect the shift $\theta$;

\item [$\cdot$] for $P\in\mathrm{ext} \ \Theta^{(d)}$, $d\in\mathbb{N}$, $P$ is $\theta^d$-ergodic;

\item [$\cdot$] if $\mathbb{E}$ is not linear, the period $p_d$ of $\mathbb{E}^{(d)}$ is euqal to $d$.
\end {itemize}

For the (strong) law of large numbers of i.i.d. random variables under sublinear expectations, see \cite {P08}, \cite{Chen16, CHW19}, \cite {Song23, Zhang23} and references therein.

In Section 4, we introduce the notion of  strong ergodicity for continuous ISEs defined on $(\Omega, \mathcal{B}_b(\Omega))$. Here, $\mathcal{B}_b(\Omega)$ is the set of bounded measurable functions defined on $(\Omega, \mathcal{F})$.  For a continuous ISE that is strongly ergodic, we show that if $\mathbb{E}$ has a periodic decomposition (\ref {PD}), then

\begin {itemize}
\item [$\cdot$] $\mathbb{E}$ has a finite period $p_{\mathbb{E}}$;

\item [$\cdot$] for $P\in\mathrm{ext} \ \Theta^{(d)}$, $d\in\mathbb{N}$, $P$ is $\theta^{p_d}$-ergodic;

\item [$\cdot$] for each $f\in\mathcal{B}_b\left(\Omega\right)$ there exists a probablity $P_f\in\Theta$, such that
\begin {align*} 
 \lim_{n\to\infty}\frac{1}{n}\sum_{k=0}^{n-1}f\circ T^{kp_{\mathbb{E}}}=\mathbb{E}[f], \quad P_f\mbox{-a.s.}
\end {align*}
\end {itemize}

We provide several examples in Section 5. Besides the example of i.i.d. sequence of random variables listed above, we show that for a continuous ISE defined on $\left(\Omega, \mathcal{B}_b\left(\Omega\right)\right)$, if $\Omega$ is countable, then $\mathbb{E}$ has a periodic decomposition. We also give an example of ISEs that does not have a periodic decomposition.

\section{Notations and Preliminaries}

 Let $(\Omega,\mathcal{F}, T)$ be a measurable system, and let $\mathcal {H}$ be a vector space of $\mathcal{F}$-measurable functions on $\Omega$ such that $1\in \mathcal{H}, f\circ T, |f|\in \mathcal{H} \ \mathrm{if} \ f\in\mathcal{H}.$  Denote by $\mathcal{I}_d$, $d\in\mathbb{N}$ the $T^d$-invariant $\sigma$-algebra, and write $\mathcal{I}$ for $\mathcal{I}_1$.

A sublinear expectation $\mathbb{E}=\sup_{P\in \Theta}\mathrm{E}_P$ on $(\Omega, \mathcal{H})$ is said to be $T$-invariant if \[\mathbb{E}[f\circ T]=\mathbb{E}[f],  \ \emph{for} \ f\in\mathcal{H}.\] Here  $\Theta$ is a family of probabilities on $(\Omega, \mathcal{F})$. Denote by $\mathcal{M} (\Omega)$ the collection of probabilities on $(\Omega,\mathcal{F}).$  In this paper, we assume that \[\Theta=\left\{ P\in \mathcal {M} (\Omega) \ | \ \mathrm{E}_P[f]\le \mathbb{E}[f], \ f\in\mathcal{H}\right\}.\]

We consider two types of regularity for sublinear expectations: continuous and regular. 

Let $\mathcal{B}_b(\Omega)$ be the collection of bounded $\mathcal{F}$-measurable functions on $\Omega$. A sublinear expectation $\mathbb{E}=\sup_{P\in \Theta}\mathrm{E}_P$ on $(\Omega, \mathcal{B}_b(\Omega))$ is said to be continuous if the associated upper probability $\mathbb{V}(A)=\sup_{P\in\Theta}P(A)$, $A\in\mathcal{F},$ is continuous from above, namely, $\lim\limits_{n\to\infty}\mathbb{V}(A_{n})=0$ for $A_n\in \mathcal{F}$ with $A_{n}\downarrow \emptyset$.

\begin{remark}\label{continuity of capacity} Let $A_{n}\in\mathcal{F}, n\ge 1$, be a sequence of disjoint sets. If $\mathbb{V}(A_{n})\ge\varepsilon$ for some $\varepsilon>0$ and all $n\ge 1$, then $\mathbb{V}$ cannot be continuous. In fact, $\mathop{\cup}\limits_{k\geq n}A_{k}\downarrow \varnothing,$ but $\mathbb{V}(\mathop{\cup}\limits_{k\geq n}A_{k})\geq\mathbb{V}(A_{n})\ge\varepsilon.$
\end{remark}

\begin{definition}
    Let $\mathbb{V}$ be a continuous upper probability on $(\Omega,\mathcal{F}).$  $\mathbb{V}$ is said to be $T$-invariant if \[\mathbb{V}(T^{-1}A)=\mathbb{V}(A), \quad A\in\mathcal{F}.\]
We say $\mathbb{V}$ is $T$-ergodic, if $\mathbb{V}$ is $T$-invariant and $\mathbb{V}(A)=0,$ or 1, for $A\in\mathcal{I}.$
\end{definition}

\begin {theorem}(\cite{Sheng-Song24})\label{main result-SS} Assume that $\mathbb{V}(A)=\sup_{P\in\Theta}P(A)$ is $T$-ergodic. Let $\Theta^{(1)}$ (resp. $\Theta^{(1)}_{\mathfrak{e}}$) be the subset of $T$-invariant (resp. ergodic) probabilities in $\Theta$. Then, the cardinality of $\Theta^{(1)}_{\mathfrak{e}}$ is finite, and $\Theta^{(1)}=\mathrm{co}\ \Theta^{(1)}_{\mathfrak{e}}$, the convex hull of $\Theta^{(1)}_{\mathfrak{e}}$.

\end {theorem}

Let $\Omega$ be a Polish space, and let $\mathcal{F}$ be the Borel $\sigma$-algebra on $\Omega$. A sublinear expectation $\mathbb{E}=\sup_{P\in \Theta}\mathrm{E}_P$ on $(\Omega, C_b(\Omega))$ is said to be regular if it is continuous from above along $C_b(\Omega)$, namely, $\lim\limits_{n\to\infty}\mathbb{E}[\phi_{n}]=0$ for $\phi_n\in C_b(\Omega)$ with $\phi_{n}\downarrow 0$.

\begin{remark}
  A sublinear expectation $\mathbb{E}=\sup_{P\in \Theta}\mathrm{E}_P$ on $(\Omega, C_b(\Omega))$ is regular if and only if $\Theta$ is weakly compact.
\end{remark}

\section{The periodic decomposition of ISEs}

In this section, we give a periodic decomposition of ISEs. First, we give the definition of periodicity of an ISE.

\begin {definition} We say a $T$-invariant sublinear expectation $\mathbb{E}$ on $(\Omega, \mathcal{H})$ is periodic  if there exists $p\in\mathbb{N}$ such that for any $f\in \mathcal{H}$
\[\mathbb{E}\left[f\circ T^{p}-f\right]=0.\] We call the smallest such integer the period of $\mathbb{E}$, denoted by $p_{\mathbb{E}}$.
\end {definition}

The following proposition provides a periodic decomposition for ISEs.

\begin {proposition} \label {perioD} For a $T$-invariant sublinear  expectation $\mathbb{E}$ on $(\Omega, \mathcal{H})$,  for any $d\in\mathbb{N}$, 
\begin {align}\mathbb{E}^{(d)}\left[f\right]:=\lim_n \frac{1}{n}\mathbb{E}\left[\sum_{k=0}^{n-1}f\circ T^{kd}\right],
  \ f\in \mathcal{H},
\end {align} is a well-defined $T$-invariant sublinear expectation on $(\Omega, \mathcal{H})$.  

1) For $d \in\mathbb{N}$, $\mathbb{E}^{(d)}\left[f+g\circ T^d-g\right]=\mathbb{E}^{(d)}\left[f\right]$, $f, g\in\mathcal{H}$. Particularly, $\mathbb{E}^{(d)}$ is periodic.
  
2) Let $\Theta^{(d)}$ be the set of probabilities on $\Omega$ dominated by  $\mathbb{E}^{(d)}$, i.e., \[\mathrm{E}_P\left[f\right]\le \mathbb{E}^{(d)}\left[f\right], \ \textmd{for} \ f\in \mathcal{H}.\]  Then, $P\in \Theta^{(d)}$ if and only if $P$ is dominated by $\mathbb{E}$ and $T^d$-invariant.

3) For $d, l\in\mathbb{N}$, $d\mid l$, we have $\mathbb{E}^{(d)}[f]\le \mathbb{E}^{(l)}[f]$, $f\in \mathcal{H}$, or equivalently, $\Theta^{(d)}\subset\Theta^{(l)}$.

\end {proposition}

\begin {proof}  For $f\in \mathcal{H}$, let $a_n:=\mathbb{E}\left[\sum_{k=0}^{n-1}f\circ T^{kd}\right]$. Then $a_{m+n}\le a_n+a_m$ since $\mathbb{E}$ is a $T$-invariant sublinear expectation. So $\mathbb{E}^{(d)}$ is well-defined. To prove the invariance of $\mathbb{E}^{(d)}$, by the definition, for $f\in \mathcal{H}$,
\begin{eqnarray*}
\mathbb{E}^{(d)}\left[f\circ T\right]&=&\lim_n \frac{1}{n}\mathbb{E}\left[\sum_{k=0}^{n-1}\left(f\circ T\right)\circ T^{kd}\right]\\
&=&\lim_n \frac{1}{n}\mathbb{E}\left[\left(\sum_{k=0}^{n-1}f\circ T^{kd}\right)\circ T\right]\\
&=&\lim_n \frac{1}{n}\mathbb{E}\left[\left(\sum_{k=0}^{n-1}f\circ T^{kd}\right)\right]=\mathbb{E}^{(d)}[f].
\end {eqnarray*}

1)  By the definition of $\mathbb{E}^{(d)}$, we have
\begin {align*}
\mathbb{E}^{(d)}\left[f+g-g\circ T^d\right]&=\lim_n\mathbb{E}\left[\frac{1}{n}\sum_{k=0}^{n-1}\left(f+g-g\circ T^d\right)\circ T^{kd}\right]\\
&=\lim_n\mathbb{E}\left[\frac{1}{n}\sum_{k=0}^{n-1}f\circ T^{kd}+\frac{1}{n}(g-g\circ T^{nd})\right]\\
&=\mathbb{E}^{(d)}\left[f\right].
\end {align*}

2) For a probability $P$ dominated by $\mathbb{E}$, if it is $T^d$-invariant, then, for any $f\in \mathcal{H}$, 
\[\mathrm{E}_P\left[f\right]=\lim_n\frac{1}{n}\textrm{E}_P\left[\sum_{k=0}^{n-1}f\circ T^{kd}\right]\le \lim_n\frac{1}{n}\mathbb{E}\left[\sum_{k=0}^{n-1}f\circ T^{kd}\right]=\mathbb{E}^{(d)}\left[f\right].\]
That is, $P$ is domonated by $\mathbb{E}^{(d)}$. 

To prove the converse statement, for a  probability $P$ dominated by $\mathbb{E}^{(d)}$ and $f\in \mathcal{H}$, we have
\[\mathrm{E}_P\left[\pm\left(f-f\circ T^d\right)\right]\le\mathbb{E}^{(d)}\left[\pm\left(f-f\circ T^d\right)\right]=0,\] which means $\mathrm{E}_P\left[f\right]=\mathrm{E}_P\left[f\circ T^d\right].$

3) follows immediately from 2).
\end {proof}

\begin {lemma} \label {relation-Theta}  If $\Theta^{(l)}\subset \Theta^{(d)}$ for some $l, d\in\mathbb{N}$, then we have
\begin {align}
\Theta^{(l)}=\Theta^{\left(\mathfrak{gcd}\left(l,d\right)\right)},
\end {align}where $\mathfrak{gcd}(l,d)$ is the greatest common divisor of $l$ and $d$.
\end {lemma}
\begin {proof} If $\mathfrak{gcd}(l,d)=d$ or $\mathfrak{gcd}(l,d)=l$, the conclusion follows from $3)$ in Proposition \ref {perioD}. So we only need to prove $\Theta^{(l)}\subset\Theta^{(\mathfrak{gcd}(l,d))}$ for the  $\mathfrak{gcd}(l,d)<d\land l$ case.

Step 1.  There exists $ m_1<d\land l$ such that $\mathfrak{gcd}(l,d)\mid m_1$ and $\Theta^{(l)}\subset\Theta^{(m_1)}$. 

If  $d<l$, assume $l \bmod d=d_1$. Then $d>d_1\ge 1$ and $\mathfrak{gcd}(l,d) \mid d_1$. For $P\in\Theta^{(l)}$, we have
\[P=P\circ T^{-l}=P\circ T^{-d_1},\]
which implies that $\Theta^{(l)}\subset\Theta^{(d_1)}$. 

If  $d>l$, assume $d \bmod l=l_1$. Then $l>l_1\ge 1$ and $\mathfrak{gcd}(l,d) \mid l_1$. For $P\in\Theta^{(l)}$, we have
\[P=P\circ T^{-d}=P\circ T^{-l_1},\]
which implies that $\Theta^{(l)}\subset\Theta^{(l_1)}$. 

Step 2. For $ m_1<d\land l$ such that $\mathfrak{gcd}(l,d)\mid m_1$ and $\Theta^{(l)}\subset\Theta^{(m_1)}$, there exists $m_1>m_2\ge 1$ such that $\mathfrak{gcd}(l,d) \mid m_2$ and $\Theta^{(l)}\subset\Theta^{(m_2)}$ if $m_1>\mathfrak{gcd}(l,d)$. 

In fact, if $m_1$ does not divide $l$, it follows from Step 1 with $d$ replaced by $m_1$ that there exists $m_1>m_2\ge 1$  such that $\mathfrak{gcd}\left(l, m_1\right)\mid m_2$, thus $\mathfrak{gcd}(l,d)\mid m_2$, and  $\Theta^{(l)}\subset\Theta^{(m_2)}$. 

If $m_1\mid l$, then $\Theta^{(m_1)}=\Theta^{(l)}\subset\Theta^{(d)}$ and $m_1$ does not divide $d$. It follows from Step 1 with $l$ replaced by $m_1$ that there exists $m_1>m_2\ge 1$ such that $\mathfrak{gcd}\left(d, m_1\right)\mid m_2$, thus $\mathfrak{gcd}(l,d)\mid m_2$, and $\Theta^{(l)}=\Theta^{(m_1)}\subset\Theta^{(m_2)}$. 

Step 3. By Step 1 and Step 2, we get $\Theta^{(l)}\subset\Theta^{(\mathfrak{gcd}(l, d))}$.
\end {proof}

\begin {corollary}\label {relation-periodicD} Let $\mathbb{E}$ be a $T$-invariant sublinear  expectation  on $(\Omega, \mathcal{H})$. Then, for any $l, d\in\mathbb{N}$, we have
\[\left(\mathbb{E}^{(d)}\right)^{(l)}=\mathbb{E}^{(\mathfrak{gcd}(l, d))},\] where $\mathfrak{gcd}(l, d)$ is the greatest common divider of $l$ and $d$.

\end {corollary}

\begin {proof}
By 2) in Proposition \ref {perioD}, we can easily get that $\left(\Theta^{(d)}\right)^{(k)}=\Theta^{(k)}$ for $k\mid d$. Noting that 
\[\left(\Theta^{(d)}\right)^{(l)}\subset \Theta^{(d)}=\left(\Theta^{(d)}\right)^{(d)},\] it follows from Lemma \ref {relation-Theta} that 
\[\left(\Theta^{(d)}\right)^{(l)}=\left(\Theta^{(d)}\right)^{(\mathfrak{gcd}(l, d))}=\Theta^{(\mathfrak{gcd}(l, d))},\] which is equivalent to
\[\left(\mathbb{E}^{(d)}\right)^{(l)}=\mathbb{E}^{(\mathfrak{gcd}(l, d))}.\] 
\end {proof}

\begin {lemma}\label {lemma-chara-period} Let $\mathbb{E}$ be a $T$-invariant sublinear expectation  on $(\Omega, \mathcal{H})$. For $d\in\mathbb{N}$, 
\begin{align}\label {chara-Period}
\mathbb{E}\left[f\circ T^{d}-f\right]=0, \ f\in\mathcal{H} \ \textrm{iff}  \ \mathbb{E}[f]=\mathbb{E}^{(d)}[f], \ f\in\mathcal{H}.
\end {align} In this case, $p_{\mathbb{E}}\mid d$.

\end {lemma}

\begin {proof} The  sufficiency of (\ref {chara-Period}) follows from 1) in Proposition \ref {perioD}. To prove the necessity, assume that $\mathbb{E}$ is periodic with $\mathbb{E}[f\circ T^d-f]=0$, for any $f\in\mathcal{H}$. Then

\begin {align*}
\mathbb{E}\left[\sum_{k=0}^{n-1}f\circ T^{kd}\right]-n\mathbb{E}\left[f\right]\le\mathbb{E}\left[\sum_{k=0}^{n-1}\left(f\circ T^{kd}-f\right)\right]\le\sum_{k=0}^{n-1}\mathbb{E}\left[\left(f\circ T^{kd}-f\right)\right]=0.
\end {align*}
Similarly, we can prove that $\mathbb{E}\left[\sum_{k=0}^{n-1}f\circ T^{kd}\right]-n\mathbb{E}\left[f\right]\ge0$, and consequently we get
\[\mathbb{E}\left[\sum_{k=0}^{n-1}f\circ T^{kd}\right]-n\mathbb{E}\left[f\right]=0,\] which implies that 
\[\mathbb{E}[f]=\mathbb{E}^{(d)}[f].\] The proof to (\ref {chara-Period}) is completed.

By (\ref {chara-Period}), we have $\mathbb{E}=\mathbb{E}^{(p_{\mathbb{E}})}.$ So, $\Theta^{(d)}=\Theta^{(p_{\mathbb{E}})}$. We conclude from \ref {relation-Theta} that $\Theta^{(p_{\mathbb{E}})}=\Theta^{(\mathfrak{gcd}(p_{\mathbb{E}}, d))}$, i.e., $\mathbb{E}^{(p_{\mathbb{E}})}=\mathbb{E}^{(\mathfrak{gcd}(p_{\mathbb{E}}, d))}$. By (\ref {chara-Period}) and the definition of the period, we have $p_{\mathbb{E}}=\mathfrak{gcd}(p_{\mathbb{E}}, d)$, which means $p_{\mathbb{E}}\mid d$.
\end {proof}

\begin {theorem}\label {Thm-chara-PeriodicD} For a $T$-invariant sublinear  expectation $\mathbb{E}$ on $(\Omega, \mathcal{H})$, write $p_d$ for the period of $\mathbb{E}^{(d)}$. Then

1) $p_d \mid d$;

2) $p_d=\inf\left\{l\in\mathbb{N}\mid \mathbb{E}^{(l)}=\mathbb{E}^{(d)}\right\};$

3) $p_d \mid p_l$ if $d\mid l$.
\end {theorem}

\begin {proof}
1) Noting that $\left(\mathbb{E}^{(d)}\right)^{(d)}=\mathbb{E}^{(d)}$ by Corollary \ref{relation-periodicD}, it follows from Lemma \ref{lemma-chara-period} that $p_d\mid d$.

2) By Lemma \ref{lemma-chara-period}, we have $p_d=\inf\left\{l\in\mathbb{N}\mid \left(\mathbb{E}^{(d)}\right)^{(l)}=\mathbb{E}^{(d)}\right\}$. It follows form Proposition \ref {relation-periodicD} that
\[p_d=\inf\left\{l\in\mathbb{N}\mid \mathbb{E}^{\left(\mathfrak{gcd}(l,d)\right)}=\mathbb{E}^{(d)}\right\}=\inf\left\{l\in\mathbb{N}\mid \mathbb{E}^{(l)}=\mathbb{E}^{(d)}\right\}.\]

3) By 2), we get $\mathbb{E}^{(p_d)}=\mathbb{E}^{(d)}\le \mathbb{E}^{(l)}=\mathbb{E}^{(p_l)}$, i.e., $\Theta^{(p_d)}\subset\Theta^{(p_l)}$. It follows from Lemma \ref{relation-Theta} that $\Theta^{(p_d)}=\Theta^{\left(\mathfrak{gcd}\left(p_d, p_l\right)\right)}$, i.e., $\mathbb{E}^{(p_d)}=\mathbb{E}^{\left(\mathfrak{gcd}\left(p_d, p_l\right)\right)}$. By 2), we get $p_d=\mathfrak{gcd}\left(p_d, p_l\right)$, i.e., $p_d \mid p_l$.
\end {proof}
\section{Strong ergodicity of ISEs}

Let $(\Omega,\mathcal{F}, T)$ be a measurable system, and let $\mathbb{E}=\sup_{P\in\Theta}\mathrm{E}_P$ be a continuous $T$-invariant sublinear expectation on $(\Omega, \mathcal{B}_b(\Omega))$. Here, $\mathcal{B}_b(\Omega)$ is the collection of bounded $\mathcal{F}$-measurable functions on $\Omega$. Set 
\begin {align*}
\hat{\Theta}:=\left\{P\in\mathcal{M}(\Omega)\mid P[A]\le\mathbb{E}[1_{A}], \ A\in\mathcal{F}\right\}.
\end {align*}
Genenrally, we have $\Theta\subset\hat{\Theta}$, but $\mathbb{V}(A):=\sup_{P\in\Theta}P(A)=\sup_{P\in\hat{\Theta}}P(A)$, $A\in\mathcal{F}$, which is $T$-invariant. Write $\Theta(T)$ (resp., $\hat{\Theta}(T)$) the set of $T$-invariant probabilities in $\Theta$ (resp., $\hat{\Theta}$).

\begin {lemma} \label {invariant-Theta} Let $\mathbb{E}$ be a continuous $T$-invariant sublinear expectation on $(\Omega, \mathcal{B}_b(\Omega))$. For each $P\in\Theta$, there exists $P'\in\Theta(T)$ such that $P=P'$ on the $T$-invariant $\sigma$-algebra $\mathcal{I}$.
\end {lemma}
It has been shown in \cite {ergodictheorem} that the same conclusion holds for $\hat{\Theta}$. The proof below is also adapted from the $\hat{\Theta}$ case.

\begin {proof}
For $P\in\Theta$, set $\mathrm{E}_{P_n}[f]=\frac{1}{n}\sum_{k=0}^{n-1}\mathrm{E}_P[f\circ T^k]$, $f\in\mathcal{B}_b(\Omega)$. Let $\Lambda$ be a Banach–Mazur limit, and set $\ell(f)=\Lambda\left(\left(\mathrm{E}_{P_n}[f]\right)\right)$, which is additive since $\Lambda$ is a linear functional. It follows from the positivity of $\Lambda$ that $\ell(f)\le\Lambda\left(\left(\mathbb{E}[f]\right)\right)=\mathbb{E}[f]$. The continuity of $\mathbb{E}$ implies that $P'(A):=\ell(1_A)$ is a probability, and it is easily to show that $\mathrm{E}_{P'}[f]=\ell(f)\le \mathbb{E}[f]$. So we conclude that $P'\in\Theta$. Noting that $P_n(A)=P(A)$ for any $A\in\mathcal{I}$ and $n\in\mathbb{N}$, we get $P'(A)=P(A)$ for each $A\in\mathcal{I}$. For any $A\in\mathcal{F}$, note  that $\limsup_{n}|P_n(A)-P_n(T^{-1}A)|=0$.  By Remark \ref {Rem-banachmazur}, we conclude that $P'$ is $T$-invariant.
\end {proof}

\begin {proposition} \label {structure-Theta} Let $\mathbb{E}$ be a continuous $T$-invariant sublinear expectation on $(\Omega, \mathcal{B}_b(\Omega))$. If $\mathbb{V}$ is $T$-ergodic, we have $\Theta(T)=\hat{\Theta}(T)$.
\end {proposition}
\begin {proof}Since $\Theta\subset\hat{\Theta}$, we get $\Theta(T)\subset\hat{\Theta}(T)$. So, we only need to prove $\Theta(T)\supset\hat{\Theta}(T)$. 
Theroem 3.3 in \cite {Sheng-Song24} shows that $\hat{\Theta}(T)$ is the convex hull of a finite set of $T$-ergodic probabilities, i.e., $\hat{\Theta}(T)=\mathrm{co}\{P_1, \cdots, P_m\}$.  Set \[K=\left\{\alpha \in \mathbb{R}^m \mid \alpha_i\ge0, \sum_{i=1}^m\alpha_i=1, \sum_{i=1}^m\alpha_i P_i\in \Theta(T)\right\},\] which is closed and
\[\left\{P_{\alpha}=\sum_{i=1}^m\alpha_i P_i\mid \alpha\in K\right\}=\Theta(T).\]
Let $\{E_k\}_{k=1}^m\subset \mathcal{I}$ be a partition of $\Omega$ such that $P_i(E_i)=1$, $i=1, \cdots, m.$ By Lemma \ref {invariant-Theta},  for each $k=1, \cdots, m$, we have $1=\mathbb{V}(E_k)=\sup_{P\in\Theta(T)}P(E_k)$. Therefore, there $\alpha^k\in K$ such that $P_{\alpha^k}(E_k)=1$. Noting that $1=P_{\alpha^k}(E_k)=\sum_{i=1}^m\alpha^k_iP_i(E_k)=\alpha^k_k$, we have $P_{\alpha^k}=P_k\in\Theta(T)$ for each $k=1, \cdots, m$. Thus, we get
$\Theta(T)=\hat{\Theta}(T).$
\end {proof}

\begin {definition} For a continuous $T$-invariant sublinear expectation $\mathbb{E}$ with finite period $p_{\mathbb{E}}$ on $(\Omega, \mathcal{B}_b(\Omega))$, we say $\mathbb{E}$ (or $\mathbb{V}$) is strongly $T$-ergodic, if $\mathbb{V}(A)=0,$ or 1, for $A\in\mathcal{I}_{p_{\mathbb{E}}}$. 
\end {definition}

\begin {proposition}\label {chara-strongErgodicity}
Let $\mathbb{E}$ be  a continuous $T$-invariant sublinear expectation  with finite period $p_{\mathbb{E}}$ on $(\Omega, \mathcal{B}_b(\Omega))$. Then $\mathbb{E}$  is strongly $T$-ergodic if and only if, for each $d\in\mathbb{N}$, $\mathbb{E}^{(d)}$ is strongly $T$-ergodic. In this case, $\Theta^{(d)}$ is the convex hull of a finite set of $T^{p_d}$-ergodic probabilities.
\end {proposition}
\begin {proof}
The sufficiency is obvious since $\mathbb{E}=\mathbb{E}^{(p_{\mathbb{E}})}$ by Lemma \ref {lemma-chara-period}.

For the necessity, it follows from Lemma \ref {relation-Theta} that $\mathbb{E}^{(d)}=\mathbb{E}^{(\mathfrak{gcd}(d, p_{\mathbb{E}}))}$ since $\mathbb{E}^{(d)}\le\mathbb{E}=\mathbb{E}^{(p_{\mathbb{E}})}$. By Theorem \ref {Thm-chara-PeriodicD}, we have $p_d\mid\mathfrak{gcd}(d, p_{\mathbb{E}})$, and consenquently, $p_d\mid p_{\mathbb{E}}$. So the strong $T$-ergodicity of $\mathbb{E}$ implies that $\mathbb{E}$ is $T^{p_d}$-ergodic. Then, Theorem \ref {main result-SS} and Proposition \ref{structure-Theta} shows that $\Theta(T^{p_d})$ is the convex hull of a finite set of $T^{p_d}$-ergodic probabilities. Noting that $\Theta^{(d)}=\Theta^{(p_d)}=\Theta(T^{p_d})$, we conclude that $\mathbb{V}^{(d)}=\mathbb{V}^{(p_d)}:=\sup_{P\in\Theta^{(p_d)}}P$ is trivial on $\mathcal{I}_{p_d}$, i.e., $\mathbb{E}^{(d)}$ is strongly ergodic.
\end {proof}

Motivated by Proposition \ref {chara-strongErgodicity}, we give the definition of strong ergodicity for genenral $T$-invariant sublinear expectation $\mathbb{E}$ on $(\Omega, \mathcal{B}_b(\Omega))$.

\begin {definition} For a continuous $T$-invariant sublinear expectation $\mathbb{E}$ on $(\Omega, \mathcal{B}_b(\Omega))$, we say $\mathbb{E}$ (or $\mathbb{V}$) is strongly $T$-ergodic, if for each $d\in\mathbb{N}$, $\mathbb{E}^{(d)}$ is strongly ergodic. 
\end {definition}
\begin {theorem}\label {strongErgor-conti} Let $\mathbb{E}$ be a continuous sublinear expectation on $(\Omega, \mathcal{B}_b(\Omega))$ and assume that $\mathbb{E}$ is strongly $T$-ergodic.  Then, there exists $d\in\mathbb{N}$ such that for any $l\in\mathbb{N}$, we have $\Theta^{(l)}\subset\Theta^{(d)}$.
\end {theorem}
\begin {proof}
Since $\mathbb{E}^{(k)}$, $k\in\mathbb{N}$, is strongly $T$-ergodic, Proposition \ref {chara-strongErgodicity} shows that $\Theta^{(k)}$ is the convex hull of a finite set of $T^{p_k}$-ergodic probabilities.

For $d, l\in\mathbb{N}$ with $d\mid l$, it follows from Proposition \ref {perioD} that $\Theta^{(d)}\subset \Theta^{(l)}$. So, the set of extreme points $\mathrm{ext}\ \Theta^{(d)}\subset \Theta^{(l)}=\mathrm{co}(\mathrm{ext}\ \Theta^{(l)})$. Write $\mathrm{ext}\ \Theta^{(d)}=\{P_j\mid j=1, \cdots, n_d\}$,  and  $\mathrm{ext}\ \Theta^{(l)}=\{Q_i\mid i=1, \cdots, n_l\}$. Here $n_k$, $k=d, l$, denotes the cardinality of $\mathrm{ext}\ \Theta^{(k)}$.  Then, for each $1\le j\le n_d$, there $\alpha^j\in\mathbb{R}^{n_l}_+$ with $\sum_{i=1}^{n_l}\alpha^j_i=1$ such that $P_j=\sum_{i=1}^{n_l}\alpha^j_iQ_i$. Since the probabilities in $\mathrm{ext}\ \Theta^{(k)}$, $k=d, l$, are mutually singular, the sets $A_j=\{i\mid \alpha^j_i>0, 1\le i\le n_l\}$, $j=1, \cdots, n_d$, are disjoint from each other but nonempty. Thus, we have $\mathrm{card}\left(\mathrm{ext}\ \Theta^{(d)}\right)\le\mathrm{card}\left(\mathrm{ext}\ \Theta^{(l)}\right)$, and the equality holds if and only if $\mathrm{ext}\ \Theta^{(d)}=\mathrm{ext} \ \Theta^{(l)}$, which is also equivalent to $\Theta^{(d)}=\Theta^{(l)}$.

If the conclusion of the theorem is wrong, it follows from Proposition \ref {perioD} that there will be a sequence $\{n_k\}$ with $n_k \mid n_{k+1}$ such that
\[\Theta^{(n_k)}\subset \Theta^{(n_{k+1})}, \quad \textmd{but} \ \Theta^{(n_k)}\neq \Theta^{(n_{k+1})}.\]
So, by the above arguments, we have $\mathrm{card}\left(\mathrm{ext} \ \Theta^{(n_k)}\right)<\mathrm{card}\left(\mathrm{ext} \ \Theta^{(n_{k+1})}\right)$, for all $k\in\mathbb{N}$, which is a contradiction by Remark \ref{continuity of capacity} since $\mathbb{E}$ is continuous.
\end {proof}

\begin {definition} For a $T$-invariant sublinear expectation $\mathbb{E}$ on $(\Omega, \mathcal{H})$, we say $\mathbb{E}$ has a periodic decomposition if \begin {align}\mathbb{E}[f]=\sup_{d\in\mathbb{N}}\mathbb{E}^{(d)}[f], f\in\mathcal{H}.
\end {align}
\end {definition}
Some examples of ISEs having a periodic decomposition are provided in the next section.

\begin {corollary} Let $\mathbb{E}$ be a continuous sublinear expectation on $\left(\Omega, \mathcal{B}_b\left(\Omega\right)\right)$ and assume that $\mathbb{E}$ is strongly $T$-ergodic.  Then, $\mathbb{E}$ has a finite period $p_{E}$ if it has a periodic decomposition. In this case, for each $P\in\Theta$ and $f\in\mathcal{B}_b\left(\Omega\right)$, we have
\begin {align} \label {ergodicity-1}
  -\mathbb{E}[-f]\leq\lim_{n\to\infty}\frac{1}{n}\sum_{k=0}^{n-1}f\circ T^{kp_{\mathbb{E}}}\leq \mathbb{E}[f], \quad P\mbox{-a.s.},
\end {align}
and for each $f\in\mathcal{B}_b\left(\Omega\right)$ there exists a probablity $P_f\in\Theta$, such that
\begin {align} \label {ergodicity-2}
 \lim_{n\to\infty}\frac{1}{n}\sum_{k=0}^{n-1}f\circ T^{kp_{\mathbb{E}}}=\mathbb{E}[f], \quad P_f\mbox{-a.s.}
\end {align}
\end {corollary}
\begin {proof}It follows from Theorem \ref {strongErgor-conti} that there exists $d\in\mathbb{N}$ such that for any $l\in\mathbb{N}$, we have $\mathbb{E}^{\left(l\right)}\le\mathbb{E}^{\left(d\right)}$.  Thus, the assumption that $\mathbb{E}$ has a periodic decomposition implies $\mathbb{E}=\mathbb{E}^{\left(d\right)}$, which has a finite   
period denoted by $p_{\mathbb{E}}$. By the strong ergodicity of $\mathbb{E}$, we conclude that $\mathbb{E}$ is $T^{p_{\mathbb{E}}}$-ergodic. Then, $\left(\ref {ergodicity-1}\right)$ follows from Theorem 3.4 in \cite {Sheng-Song24}.

By Proposition \ref {chara-strongErgodicity}, $\Theta^{d}$ is the convex hull of a finite set of $T^{p_{\mathbb{E}}}$-ergodic.  So, for any $f\in\mathcal{B}_b\left(\Omega\right)$, there exists a $T^{p_{\mathbb{E}}}$-ergodic probability $P\in\Theta^{\left(d\right)}$ such that $\mathrm{E}_{P_f}[f]=\mathbb{E}^{d}[f]=\mathbb{E}[f]$. Then, Birkhoff’s
ergodic theorem gives  (\ref {ergodicity-2}).
\end {proof}

\section{Examples}

We first give two typical examples of ISEs which have a periodic decomposition.
\begin {example} For $\Omega=\mathbb{R}^{\infty}$ endowed with the metric $\rho(x,y)=\sum_{k\in\mathbb{N}}\frac{1}{2^n}\left(\left|x_k-y_k\right|\land 1\right)$, let $\mathbb{E}$ be a regular sublinear expectation on $\left(\Omega, C_b(\Omega)\right)$ such that $\xi_n\left(\omega\right)=\omega_n$, $n\in\mathbb{N}$, are independent and identically distributed. 

1) $\mathbb{E}$ has a periodic decomposition with respect to the shift
\[\theta\left(x_1, x_2, \cdots\right)=\left(x_2, x_3, \cdots\right).\]

2) For $P\in\mathrm{ext} \ \Theta^{(d)}$, $d\in\mathbb{N}$, $P$ is $\theta^d$-ergodic.

3) If $\mathbb{E}$ is not linear, the period $p_d$ of $\mathbb{E}^{(d)}$ is euqal to $d$.
\end {example}
\begin {proof}
1) For $f(\omega)=\phi(\omega_1, \cdots, \omega_d)$ with $\phi\in C_{b, \mathrm{Lip}}(\mathbb{R}^d)$, we get that
\begin {eqnarray*}\mathbb{E}^{(d)}[f]=\lim_n\frac{1}{n}\mathbb{E}\left[\sum_{k=0}^{n-1}\phi(\xi_{kd+1}, \cdots, \xi_{(k+1)d})\right]=\mathbb{E}[f]
\end {eqnarray*} since $\{\xi_n\}_{n\ge1}$ are independent and identically distributed. So, we get $\mathbb{E}[f]=\sup_{d\in\mathbb{N}}\mathbb{E}^{(d)}[f]$, which holds for all $f\in C_b(\Omega)$ by approximation arguments.

2)  It follows from Proposition \ref {perioD}  that probabilities in $\Theta^{(d)}$ are $\theta^d$-invairant. For $P\in\mathrm{ext} \ \Theta^{(d)}$, by Birkhoff’s ergodic theorem, for any $f\in C_b(\Omega)$,
\[\lim_{n\to\infty}\frac{1}{n}\sum_{k=0}^{n-1}f\circ \theta^{kd}=\mathrm{E}_{P}[f\mid \mathcal{I}_d].\]

We claim that for $f\in C_b(\Omega)$, 
\begin {align}\label {IID-theta}\mathrm{E}_{P}[f\mid \mathcal{I}_d]\le \mathbb{E}[f], \quad P\mathrm{-a.s.}
\end {align}
If $P$ is not $\theta^d$-ergodic, there will be $A\in\mathcal{I}_d$ such that $0<P(A)<1$. Then $P_A(\cdot)=P(\cdot\mid A)$ and $P_{A^c}(\cdot)=P(\cdot\mid A^c)$ are both $\theta^d$-invairant since $A\in \mathcal{I}_d$. (\ref {IID-theta}) shows that $P_A, P_{A^c}$ are dominated by $\mathbb{E}$, which implies that $P_A, P_{A^c}$ belongs to $\Theta^{(d)}$ by Proposition \ref {perioD}. This is a contradiction since $P=P(A)P_A+P(A^c)P_{A^c}$,

We now prove (\ref {IID-theta}). Actually, fix $f(\omega)=\phi(\omega_1, \cdots, \omega_{ld})$ with $\phi\in C_{b,\mathrm{Lip}}(\mathbb{R}^{ld})$, $l\in\mathbb{N}$. For any $n\in\mathbb{N}$, writing $n=ml+r$ with $m, r\in\mathbb{Z}_+$, $0\le r <l$, we have
\[S_{n}^{(d)}f:=\sum_{k=0}^{n-1}f\circ\theta^{kd}=\sum_{i=0}^{l-1}\sum_{k=0}^{m-1}f\circ\theta^{(kl+i)d}+\sum_{j=0}^{r-1}f\circ\theta^{(ml+j)d}.\] 
Thus,
\[\frac{1}{n}S_{n}^{(d)}f=\frac{ml}{n}\left(\frac{1}{l}\sum_{i=0}^{l-1}\left(\frac{1}{m}\sum_{k=0}^{m-1}f\circ\theta^{(kl+i)d}\right)\right)+\frac{r}{n}\left(\frac{1}{r}\sum_{j=0}^{r-1}f\circ\theta^{(ml+j)d}\right),\]
and consequently,
\[\left(\frac{1}{n}S^{(d)}_{n}f-\mathbb{E}[f]\right)^+\le \frac{m}{n}  \sum_{i=0}^{l-1}  \left(\frac{1}{m}\sum_{k=0}^{m-1}f\circ\theta^{(kl+i)d}-\mathbb{E}[f]\right)^++ \frac{2r}{n}\|f\|_{\infty}.\]
It follows from the weak law of large numbers under sublinear expectations that
\[\lim_{n\to\infty} \mathbb{E}\left[\left(\frac{1}{n}S^{(d)}_{n}f-\mathbb{E}[f]\right)^+\right]=0\] 
since $(f\circ\theta^{(kl+i)d})_{k\ge0}$ is independent and identically distributed under $\mathbb{E}$ for each $0\le i<l$. For the case $f\in C_b(\Omega)$, we can get the desired result by approximation arguments. Thus,

\[\mathrm{E}_P\left[\left(\mathrm{E}_{P}[f\mid \mathcal{I}_d]-\mathbb{E}[f]\right)^+\right]\le\mathbb{E}\left[\liminf_{n\to\infty}\left(\frac{1}{n}S^{(d)}_{n}f-\mathbb{E}[f]\right)^+\right]\le \lim_{n\to\infty} \mathbb{E}\left[\left(\frac{1}{n}S^{(d)}_{n}f-\mathbb{E}[f]\right)^+\right]=0.\]

3) By Theorem \ref {Thm-chara-PeriodicD},  it suffices to prove that for any $l\mid d$, $l<d$ there exists $f\in C_b(\Omega)$ such that $\mathbb{E}^{(d)}[f-f\circ \theta^{l}]\neq0.$ 

Actually, there exists $\phi\in C_b(\mathbb{R})$ such that $\overline{\mu}=\mathbb{E}[\phi(\xi_1)]>-\mathbb{E}[-\phi(\xi_1)]=\underline{\mu}$. Set $f=\sum_{k=1}^{l} \phi(\xi_k)$. Then, since $2l\le d$,  we have
\[\mathbb{E}^{(d)}[f-f\circ\theta^{l}]= \mathbb{E}[f-f\circ\theta^{l}]=l(\overline{\mu}-\underline{\mu})>0.\]
\end {proof}

\begin {example} \label {period-count}
Let $\left(\Omega,\mathcal{F}, T\right)$ be a measurable system, and let $\mathbb{E}$ be a continuous $T$-invariant sublinear expectation on $\left(\Omega, \mathcal{B}_b\left(\Omega\right)\right)$. If $\Omega$ is countable, $\mathbb{E}$ has a periodic decomposition.
\end {example}
\begin {proof}
Let $\Theta=\left\{P\in\mathcal{M}\left(\Omega\right)\mid \mathrm{E}_P[f]\le\mathbb{E}[f], \ f\in\mathcal{B}_b\left(\Omega\right)\right\}$. It follows from Lemma \ref{invariant-Theta} that for any $P\in \Theta$, there is a $T$-invariant probability $P_0\in\Theta$ such that $P=P_0$ on $\mathcal{I}$. 

We shall further prove that $P\ll P_0$.  Otherwise, there is $N\in \mathcal{F}$ such that $P_0\left(N\right)=0$, but $P\left(N\right)>0.$ Note that 
\[A_n:=\cap_{k=n}^{\infty}T^{-k} N^c \uparrow \cup_{n=1}^{\infty}A_n=:A, \quad T^{-1}A_n=A_{n+1}.\]
So $A$ is $T$-invariant,  and thus we get $P\left(A\right)=P_0\left(A\right)=1.$ Noting that $A\cap A_n^c \downarrow \emptyset$, but
\[\mathbb{V}\left(A\cap A_n^c\right)=\mathbb{V}\left(T^{-n}A\cap T^{-n}A_0^c\right)=\mathbb{V}\left(A\cap A_0^c\right)\ge P\left(A\cap A_0^c\right)=P\left(A_0^c\right)\ge P\left(N\right)>0,\] which is a contradition.

Let $\Omega_0:=\{x\in \Omega \mid P_0\left(\{x\}\right)>0\}$.  Set $\Omega_{0,d}:=\{x\in \Omega_0\mid T^dx=x\}$, which is $T$-invariant ($P_0$-a.s.). Furthermore, any subset of $\Omega_{0,d}$ is $T^d$-invariant.  By Poincar$\acute{\textmd{e}}$'s reccurrence theorem, any $x\in \Omega_0$ is periodic, i.e., $T^nx=x$, for some $n\in\mathbb{N}$.  Then, for any $\varepsilon>0,$ there exists $d_{\varepsilon}\in\mathbb{N}$ such that $P_0\left(\Omega_{0,d_{\varepsilon}}\right)>1-\varepsilon.$

 Set $P^{\varepsilon}_n=\frac{1}{n}\sum_{k=0}^{n-1}P\circ T^{-kd_{\varepsilon}}$.
Then, for $A\in\mathcal{F}$, 
\begin {align*}
P^{\varepsilon}_n\left(A\right)=&P^{\varepsilon}_n\left(A\cap\Omega_{0, d_{\varepsilon}}\right)+P^{\varepsilon}_n\left(A\cap\Omega_{0, d_{\varepsilon}}^c\right)\\
=&P\left(A\cap\Omega_{0, d_{\varepsilon}}\right)+P^{\varepsilon}_n\left(A\cap\Omega_{0, d_{\varepsilon}}^c\right)\\
=&P\left(A\right)-P\left(A\cap\Omega_{0, d_{\varepsilon}}^c\right)+P^{\varepsilon}_n\left(A\cap\Omega_{0, d_{\varepsilon}}^c\right),
\end {align*}
where the second equality holds since $A\cap\Omega_{0, d_{\varepsilon}}$ is $T^{d_{\varepsilon}}$-invariant. Note that \[\left|P\left(A\cap\Omega_{0, d_{\varepsilon}}^c\right)-P^{\varepsilon}_n\left(A\cap\Omega_{0, d_{\varepsilon}}^c\right)\right|\le \max\{P\left(\Omega_{0, d_{\varepsilon}}^c\right),P^{\varepsilon}_n\left(\Omega_{0, d_{\varepsilon}}^c\right)\}= P_0\left(\Omega_{0, d_{\varepsilon}}^c\right)< \varepsilon.\]
The last equality holds since $\Omega_{0, d_{\varepsilon}}^c$ is $T$-invariant and $P=P_0$ on $\mathcal{I}$. Thus, 
\[\left|P^{\varepsilon}_n(A)-P(A)\right|<\varepsilon, \quad \emph{for} \ A\in\mathcal{F}.\]

By similar arguments as those in the proof to Lemma \ref {invariant-Theta},  we can get  a $T^{d_{\varepsilon}}$-invariant probability $P^{\varepsilon}\in\Theta$ as a Banach–Mazur limit of $P_n^{\varepsilon}$. Therefore, 
\[\left\|P^{\varepsilon}-P\right\|_{\mathrm{TV}}\le\varepsilon.\]
Thus, for $f\in\mathcal{B}_b(\Omega)$,
\[\mathbb{E}^{(d_{\varepsilon})}[f]\ge \mathrm{E}_{P^{\varepsilon}}[f]\ge \mathrm{E}_P[f]-2\varepsilon\|f\|_{\infty}.\]
So, for any $f\in\mathcal{B}_b(\Omega)$ and $P\in\Theta$,\[\sup_{d\in\mathbb{N}}\mathbb{E}^{(d)}[f]\ge \mathrm{E}_P[f],\]
which means $\mathbb{E}$ has a periodic decomposition.
\end {proof}

Below is an example of ISE which does not have a periodic decomposition.

\begin {example}
Let $K=\{z\in\mathbb{C}: |z|=1\}$, let $\mathcal{B}$ be the $\sigma$-algebra of Borel subsets of $K$ and let $m$ be Haar measure. Let $a\in K$ and define $T: K\rightarrow K$ by $Tz=az$. Then $T$ is measure-preserving since $m$ is Haar measure. $T$ is (uniquely) ergodic iff $a$ is not a root of unitary,  which is the only case we consider unless specified otherwise.

\vskip 0.1 cm

 Let $\phi$ be a nonnegative $\mathcal{B}$-measurable fucntion on $K$ such  that $m(\phi)=1$. Set $\phi_k(x)=\phi(a^{k}x)$ and $\mu_k=\phi_k.m$, $k\in\mathbb{Z}$, which are all probabilities since $m$ is $T$-invariant. Since $\mu_k\circ T^{-1}=\mu_{k-1}$, $\mathbb{E}[f]=\sup_{k\in \mathbb{Z}}\mathrm{E}_{\mu_k}[f]$, $f\in\mathcal{B}_b(\Omega)$, is $T$-invariant.  

Let $\Theta$ be the set of probabilities dominated by $\mathbb{E}$. Since $m$ is the unique $T$-invariant probability, it follows from Lemma \ref {invariant-Theta} that $m\in\Theta$ and $\mu_k=m$ on $\mathcal{I}$, $k\in\mathbb{Z}$, and thus $\mathbb{V}=\sup_{\mu\in\Theta}\mu$ is trivial on $\mathcal{I}$. Here, we give a direct proof by Birkhoff’s ergodic theorem. 

In fact, \[\mathrm{E}_m[f]=\lim_{n\to\infty}\mathrm{E}_m\left[f\frac{1}{n}\sum_{k=0}^{n-1}\phi\circ T^k\right]=\lim_{n\to\infty}\frac{1}{n}\sum_{k=0}^{n-1}\mathrm{E}_{\mu_k}[f]\le\mathbb{E}[f],\] which means $m\in\Theta.$ For $A\in\mathcal{I}$, $k\in\mathbb{Z}$, 
\[\mu_k(A)=\lim_{n\to\infty}\mathrm{E}_m\left[\phi\circ T^k\frac{1}{n}\sum_{i=0}^{n-1}1_A\circ T^{i}\right]=\lim_{n\to\infty}\mathrm{E}_m\left[\frac{1}{n}\sum_{i=0}^{n-1}\phi\circ T^{k-i}1_A\right]=m(A).\]

For $A\in\mathcal{B}$, $\mathbb{V}(A)=\sup_{k\in\mathbb{Z}}\mathrm{E}_m\left[\phi\circ T^k 1_A\right]=\sup_{k\in\mathbb{Z}}\mathrm{E}_m\left[\phi1_{T^kA}\right]$. Since $m(T^kA)=m(A)$, for any $\varepsilon>0$, there exists $\delta>0$ such that for any $A\in\mathcal{B}$ with $m(A)<\delta$, we have $\mathbb{V}(A)<\varepsilon$. So $\mathbb{V}$ is continuous.

\vskip 0.3 cm

If $a$ is not a root of unitary, neither is $a^d$, for any $d\in\mathbb{N}$. So, $m$ is the unique $T^d$-invariant measure on $K$,  i.e., $\mathbb{E}^{(d)}=\mathrm{E}_m$, and thus $\mathbb{E}$ does not have a periodic decomposition.

\end {example}

\section*{Acknowledgements}
 Song Y. is financially supported by National Key R\&D Program of China (No. 2020YFA0712700 \& No. 2018YFA0703901).

\section{Appendix}

Banach-Mazur limits (see \emph{Infinite Dimensional Analysis} by Aliprantis and Border, 2005 \cite{infiniteanalysis}) play an important role in the study of invariant capacities. Therefore, we shall give a brief introduction to it.

\begin{definition}\label{banachmazur}
    Let $\ell_{\infty}$ be the space of all bounded sequences. A positive linear functional $\Lambda: \ell_{\infty} \rightarrow \mathbb{R}$ is a Banach-Mazur limit if
    \begin{enumerate}
        \item $\quad \Lambda(\boldsymbol{e})=1$, where $\boldsymbol{e}=(1,1,1, \ldots)$,
        \item $\quad \Lambda\left(x_1, x_2, \ldots\right)=\Lambda\left(x_2, x_3, \ldots\right)$ for each $\left(x_1, x_2, \ldots\right) \in \ell_{\infty}$.
    \end{enumerate}
\end{definition}

\begin{lemma}
    If $\Lambda$ is a Banach-Mazur limit, then
$$
\liminf _{n \rightarrow \infty} x_n \leqslant \Lambda(x) \leqslant \limsup _{n \rightarrow \infty} x_n
$$
for each $x=\left(x_1, x_2, \ldots\right) \in \ell_{\infty}$. In particular, $\Lambda(x)=\lim\limits_{n \rightarrow \infty} x_n$ for each convergent sequence $x$ (so every Banach-Mazur limit is an extension of the limit functional).
\end{lemma}

\begin {remark} \label {Rem-banachmazur}For any $x, y\in \ell_{\infty}$,
\[|\Lambda(x)-\Lambda(y)|\le\lim_n\sup_{k\ge n} |x_k-y_k|.\]
\end {remark}

\begin{theorem}
    Banach–Mazur limits exist.
\end{theorem}

\renewcommand{\refname}{\large References}{\normalsize \ }

\end{document}